\documentclass[12pt, english, a4paper]{article} 
\usepackage[latin9]{inputenc}
\usepackage[T1]{fontenc}
\usepackage{graphicx}
\usepackage{amsthm}
\usepackage{amsfonts}
\usepackage{amsmath}
\usepackage{amssymb}
\usepackage{makecell}
\usepackage{comment}
\usepackage{cases}
\usepackage{mathrsfs}
\usepackage{fullpage}
\usepackage{times}
\usepackage[usenames]{color}
\usepackage[dvipsnames]{xcolor}
\usepackage{hyperref}

\newcommand{\sub}{\mathrm{sub}}

\newtheorem{theorem}{Theorem}
\numberwithin{theorem}{section}

\newtheorem{conjecture}[theorem]{Conjecture}
\newtheorem{question}[theorem]{Question}
\newtheorem{remark}[theorem]{Remark}

\newtheorem{lemma}[theorem]{Lemma}

\begin{document}
\title{$q$-analogues of Fisher's inequality and oddtown theorem}






\author{
Hiranya Kishore Dey \\ 
Department of Mathematics\\
Indian Institute of Science, Bangalore\\
Bangalore 560 012, India.\\
email: hiranyadey@iisc.ac.in
}



\maketitle

\begin{abstract}

\medskip 
\noindent 
A classical result in design theory, known as Fisher's inequality, states that if every pair of clubs in a town shares the same number of members, then the number of clubs cannot exceed the number of inhabitants in the town. In this short note, we establish a 
$q$-analogue of Fisher's inequality. Additionally, we present a 
$q$-analogue of the oddtown theorem for the case when 
$q$ is an odd prime power.


\end{abstract}

	
\section{Introduction}
\label{sec:intro}

Finding $q$-analogues of well-known problems in extremal set theory has been an active area of research. 
Some classical results in this field for which 
$q$-analogues have been studied include Bollob\'{a}s theorem \cite{bollobas}, Erd\H{o}s-Ko-Rado theorem \cite{katona, erdos-ko-rado} and Hilton-Milner theorem \cite{hilton-milner}. 
A $q$-analogue 
of the Erd\H{o}s-Ko-Rado theorem was established by Hsieh \cite{hsieh} and independently by Frankl and Wilson \cite{frankl-wilson}. Blokhuis et. al. \cite{blokhuis-et-al} obtained a $q$-analogue of the Hilton-Milner theorem, while Lov\'{a}sz \cite{lovasz} derived a $q$-analogue of the Bollob\'{a}s theorem.

In this paper, we obtain $q$-analogues of two other well-known results in extremal set theory:  Fisher's inequality \cite{fisher} and oddtown theorem \cite{berlekamp}. To the best of our knowledge, these 
$q$-analogues have not previously been explored in the literature.

For a positive integer $n$, let $[n]$ denote the set $\{1,\dots,n\}$ and its $q$-analogue is defined as 
\[ [n]_q:=\frac{1-q^n}{1-q}=1+q+q^2+\hdots+q^{n-1}. \] 
 It is easy to see that 
 \begin{equation} 
 \lim_{q \to 1} \  [n]_q = n.\label{eqn:lim}
\end{equation}
The $q$-factorial of a positive integer $n$, denoted by $[n]_q!$, is defined as 
\[ [n]_q ! : = \prod_{i=1}^n  [i]_q.\]
 Moreover, the $q$-binomial 
coefficient is defined as
\[\binom{n}{k}_q:= \frac{[n]_q!}{[k]_q ! [n-k]_q !}.\] 

Let $F_q$ denote the field in $q$ elements and consider the $n$ dimensional vector space $F_q^n$. 
Let $\sub(F_q^n)$ denote the set containing all the subspaces of $F_q^n$. It is well-known that \cite{goldman-rota} the number of $k$ dimensional subspaces of $F_q^n$ is $\binom{n}{k}_q$ and hence
\[ |\sub(F_q^n)| = \displaystyle \sum _{k=0}^n \binom{n}{k}_q.\]

\subsection{Fisher's inequality}

The well-known Fisher's inequality states the following: 

\begin{theorem}[Fisher] 
\label{thm:Fischer}
Let $k$ be a positive integer and let $A_1, \dots, A_m$ be distinct subsets of $[n]$. If  $|A_i \cap A_j|=k$ for every $1 \leq  i < j \leq m$, then $m \leq n$.  
\end{theorem}

Theorem \ref{thm:Fischer} was first proved by Fisher \cite{fisher} when $k=1$ and 
all the sets $A_i$ have the same size. 
Bose \cite{bose} used linear algebraic argument to solve the above combinatorial 
problem when all the sets are of same size. 
In \cite{debruijn-erdos},  de Bruijn and Erd\H{o}s relaxed the uniformity condition for the sets $A_i$.  The
first proof of the general form of the Fisher's Inequality was given by Majumdar \cite{majumdar}
using linear algebraic methods. Mathew and Mishra \cite{mathew-mishra} gave a simple, counting based proof of Fisher's Inequality that does
not use any tools from linear algebra. 
See the references 
\cite{babai, isbel, majumdar, woodall} for more on Fisher's inequality. 
In this short note, 
our first result is the following $q$-analogue of Fisher's inequality. 

\begin{theorem} 
\label{thm:Fischer-q}
Let $k$ be a positive integer and let $\mathcal{F} \subseteq \sub(F_q^n)$ be such that $\dim(A \cap B)=k$ for all 
$A, B \in \mathcal{F}$ with $A \neq B$. Then, $|\mathcal{F}| \leq [n]_q$.
\end{theorem}

Consider the following family 
\begin{equation}
\mathcal{F}_1 = \{  U \in \sub({F}_q^n): \dim(U)=1 \}. 
\label{eqn:defn-F1}
\end{equation} 
The intersection of any two of members from $\mathcal{F}_1$ is clearly the zero subspace. 
Thus, the upper bound in Theorem \ref{thm:Fischer-q} is the best possible upper bound. From \eqref{eqn:lim}, we note that the upper bound in Theorem 
\ref{thm:Fischer} is the limit case of the upper bound in Theorem \ref{thm:Fischer-q} when $q$ approaches $1$. 

\subsection{Oddtown theorem} 

The oddtown problem is a result which highlight the linear algebra method \cite{babai-frankl} in
extremal combinatorics. There have been numerous extensions of this result in the literature \cite{deza-frankl-singhi, frankl-odlyzko, sudakov-vieira, szabo, vu}. 

Let $\mathcal{A} = \{ A_1, \dots, A_m \}$ be a family of subsets of $[n]$. We say that $\mathcal{A}$ is an oddtown if all its sets have
odd size and
$|A_i \cap A_j|$ is even for $ 1 \leq i < j \leq m$. 
Answering a question of 
Erd\H{o}s, 
Berlekamp \cite{berlekamp} and Graver \cite{Graver} independently proved the following oddtown theorem.

\begin{theorem}[Oddtown theorem]
\label{thm:oddtown}
Let $\mathcal{F} \subseteq 2^{[n]}$ be such that $|A|$ is odd for all $A \in \mathcal{F}$ and $|A \cap B|$ is even for all $A, B \in \mathcal{F}$ with $A \neq B$. Then, we have $|\mathcal{F}| \leq n$. 
\end{theorem}

In this paper, we prove the following $q$-analogue of the oddtown theorem. 

\begin{theorem}
\label{thm:oddtown-qanalogue}
Let $\mathcal{F} \subseteq \sub(F_q^n)$ be such that $\dim(A)$ is odd for all $A \in \mathcal{F}$ and $\dim(A \cap B)$ is even for all $A, B \in \mathcal{F}$ with $A \neq B$. 
Then, 
 if $q$ is an odd prime power, 
 we have 
$|\mathcal{F}| \leq [n]_q$. 
\end{theorem}


Again, the family $\mathcal{F}_1$, as defined in \eqref{eqn:defn-F1}, shows that the upper bound in Theorem \ref{thm:oddtown-qanalogue} is the best possible upper bound. 

The following result is analogous to the oddtown theorem but it swiches the parity conditions. See \cite[Exercise 1.1.5]{babai-frankl} for reference.  

\begin{theorem}[Reverse oddtown theorem]
\label{thm:oddtown-rev}
Let $\mathcal{F} \subseteq 2^{[n]}$ be such that $|A|$ is even for all $A \in \mathcal{F}$ and $|A \cap B|$ is odd for all $A, B \in \mathcal{F}$ with $A \neq B$. Then, we have $|\mathcal{F}| \leq n$ when $n$ is odd and $|\mathcal{F}| \leq n-1$ when $n$ is even. 
\end{theorem}

We have the following $q$-analogue of Theorem \ref{thm:oddtown-rev}.  

\begin{theorem}
\label{thm:oddtown-qanalogue-rev}
Let $\mathcal{F} \subseteq \sub(F_q^n)$ be such that $\dim(A)$ is even for all $A \in \mathcal{F}$ and $\dim(A \cap B)$ is odd for all $A, B \in \mathcal{F}$ with $A \neq B$. Then, if $q$ is an odd prime power, we have $|\mathcal{F}| \leq [n]_q$ when $n$ is odd and $|\mathcal{F}| \leq [n]_q-1$ when $n$ is even. 
\end{theorem}

\begin{remark}
\label{rem:rev-odd}
When $n$ is odd, consider the family \begin{equation}
    \mathcal{F}_2 = \{  U \in \sub({F}_q^n): \dim(U)=n-1 \}. 
    \label{eqn:defn-F1}
    \end{equation}
Each $A \in \mathcal{F}_2$ has even dimension. For $A, B \in \mathcal{F}_2$ with $A \neq B$, clearly $\dim(A \cap B) = n-2$, which is odd. The cardinality of $\mathcal{F}_2$ is $\binom{n}{n-1}_q$ which is the same as $[n]_q$. Hence, when $n$ is odd, our result is the best possible. 

When $n$ is even, let $\{e_1, \dots, e_n \}$ denote the standard basis of $F_q^n$ and consider the $n-1$ dimensional subspace
$W$, generated by the vectors $e_1, \dots, e_{n-1} $. Let $\mathcal{F}_3$ denote the family containing all the $n-2$ dimensional subspaces of $W$. Then, $\mathcal{F}_3$ is 
a family of size $\binom{n-1}{n-2}_q=[n-1]_q$. We note that both 
\begin{equation}  
\lim_{q \to 1} \ [n-1]_q = n-1  \hspace{5 mm}  \mbox{ and } \hspace{5 mm} \lim_{q \to 1} 
\ [n]_q -1=n-1.
\label{eqn:rev-odd-ex}
\end{equation} 
But $[n]_q-1-([n-1]_q-1)=q^{n-1}-1$ is large when $n$ is large. 
\end{remark}

Based on Remark \ref{rem:rev-odd} and limited data, we conjecture the following stronger upper bound when $n$ is even. 

\begin{conjecture}
\label{conj:reverse-oddtown-q-oddprimepower}
Let $n$ be an even positive integer and let $\mathcal{F} \subseteq \sub(F_q^n)$ be such that $\dim(A)$ is even for all $A \in \mathcal{F}$ and $\dim(A \cap B)$ is odd for all $A, B \in \mathcal{F}$ and $A \neq B$. Then, if $q$ is an odd prime power, we have $|\mathcal{F}| \leq [n-1]_q$. 
\end{conjecture}




\section{Proof of Theorem \ref{thm:Fischer-q}}  

Bose \cite{bose} and Majumdar \cite{majumdar}
used linear algebra and associated incidence vectors to each set to prove Fisher's inequality. In the same spirit, we associate incidence vectors to each subspace to prove  Theorem \ref{thm:Fischer-q}. 

The number of one-dimensional subspaces of $F_q^n$  is $[n]_q$. Let $W_1, \dots, W_{[n]_q}$ denote all the one dimensional subspaces. From each subspace $W_i$, we choose one non-zero vector $v_i$. Then, for any $i \neq j$, $v_i$ and $v_j$ are not scalar multiple of one another. 

For any subspace $A$, we define the incidence vector, corresponding to the subspace $A$, to be a $0/1$ vector (of length $[n]_q$) as follows:

\begin{equation}
\label{eqn:defn-inc-vec}
(f_A)_j = \begin{cases}
1 \hspace{3 mm} \mbox{ if } v_j \in A, \\
0 \hspace{3 mm} \mbox{ if } v_j \notin A.
\end{cases}
\end{equation}

For two vectors $x,y \in \mathbb{R}^{[n]_q}$, let $\langle x, y \rangle $ denote their scalar product. The following lemma connects the scalar product of the incidence vectors of two subspaces with the dimension of their intersection. 

\begin{lemma}
\label{lem:imp-1}
For two subspaces $A, B$ of $F_q^n$ ($A$ and $B$ can be equal)  with $\dim(A \cap B) \geq 1$, 
\[ \langle f_A, f_B \rangle =[ \dim(A \cap B)]_q. \]
\end{lemma}

\begin{proof}
Let $\dim(A \cap B)=t$. Since $A \cap B$ is of dimension $t$, the subspace $A \cap B$ is generated by $t$ linearly independent vectors, say, $\alpha_1,  \dots, \alpha_t$. As $W_1,  \dots, W_{[n]_q}$ are all the one dimensional subspaces $F_q^n$ have, and for each subspace $W_i$, we have chosen one non-zero vector $v_i$ from $W_i$, for every $1 \leq j \leq t$, the vector $\alpha_j$ must be scalar multiple of some $v_{i}$. 


Let $x$ be a linear combination of the vectors $ \alpha_1, \dots, \alpha_t.$ That is, \[x= \sum_{k=1}^{t} a_k \alpha_{k}\] for some scalars $a_k \in F_q$. Now, $\alpha_k \in A \cap B$ for $1 \leq k \leq t$. So, $x \in A \cap B$. So, $x$ must be scalar multiple of some vector from $\{v_1, v_2, \dots, v_{[n]_q} \}$. Let, $x$ be a scalar multiple of the vector $v_x$. Then, \[ (f_A)_x=1=(f_B)_x.\] 
Thus, for any vector which is a linear combination of $\alpha_1,  \dots, \alpha_t$, the corresponding entry in $f_A$ and $f_B$ are both $1$. 
Threfore, we have $\langle f_A, f_B \rangle \geq [ t]_q$. 

We now show that $\langle f_A, f_B \rangle$ cannot be strictly more that $[ t]_q$. We assume the contrary, that is, let $\langle f_A, f_B \rangle > [ t]_q$. Then there exists at least $[ t]_q+1$ indices $i$ for which both $(f_A)_i=1$ and $(f_B)_i=1$. But $A \cap B$ is of dimension $t$ and hence it has $[t]_q$ one-dimensional subspaces, contradiction. 
\end{proof}

\begin{lemma}
\label{lem:imp-2}
For two subspaces $A, B$ of $F_q^n$ with $\dim(A \cap B) =0$, 
\[ \langle f_A, f_B \rangle =0. \]
\end{lemma}

\begin{proof}
 As $\dim(A \cap B)=0$, there is no non-zero vector in $A \cap B$. So for every $1 \leq j \leq [n]_q$, at least one of $(f_A)_j$ and  $(f_B)_j$ is zero, completing the proof.
\end{proof}

We are now in a position to prove Theorem \ref{thm:Fischer-q}.

\begin{proof}[Proof of Theorem \ref{thm:Fischer-q}]
Let $\mathcal{F} \subseteq \sub(F_q^n)$ be a family such that $\dim(A \cap B)=k$ for all $A, B \in \mathcal{F}$ with $A \neq B$ and some fixed $1 \leq k \leq n$.
For any subspace $A$, let $f_A$ denote the incidence vector corresponding to the subspace $A$, as defined in \eqref{eqn:defn-inc-vec}. We will show that the vectors $\{f_A : A \in \mathcal{F} \}$ are linearly independent. We assume the contrary, that there exists a linear relation 
\[ \sum _{A \in \mathcal{F}} \lambda_A f _A= 0_v\] 
with not all coefficients $\lambda_A $ being zero where $0_v$ denotes the zero vector. Then, we have 
\begin{equation}
\label{eqn:1}
\left( \sum _{A \in \mathcal{F}} \lambda_A f _A \right)\left( \sum _{A \in \mathcal{F}} \lambda_A f _A \right) = 0.
\end{equation}
Using Lemma \ref{lem:imp-1}, 
we have 
\begin{align*}
    \left( \sum _{A \in \mathcal{F}} \lambda_A f _A \right)\left( \sum _{A \in \mathcal{F}} \lambda_A f _A \right)  = & \sum _{A \in \mathcal{F}} \lambda_A ^2 \langle f_A, f_A \rangle + 2 \sum _{A, B \in \mathcal{F}, A \neq B}   \lambda_A \lambda_B \langle f_A, f_B \rangle \\
     = &  \sum _{A \in \mathcal{F}} \lambda_A ^2 [\dim(A)]_q + 2 \sum _{A, B \in \mathcal{F}, A \neq B}   \lambda_A \lambda_B  [\dim(A \cap B)]_q.
\end{align*}
As $\dim(A \cap B)=k$ for all $A, B \in \mathcal{F}$ with $A \neq B$, we have 
\begin{align}
  \left( \sum _{A \in \mathcal{F}} \lambda_A f _A \right)\left( \sum _{A \in \mathcal{F}} \lambda_A f _A \right)    = & \sum _{A \in \mathcal{F}} \lambda_A ^2 [\dim(A)]_q + 2 \sum _{A, B \in \mathcal{F}, A \neq B}   \lambda_A \lambda_B  [k]_q.  \nonumber 
  \end{align} 
We rewrite this as 
\begin{equation}
  \left( \sum _{A \in \mathcal{F}} \lambda_A f _A \right)\left( \sum _{A \in \mathcal{F}} \lambda_A f _A \right)    =  \sum _{A \in \mathcal{F}} \lambda_A ^2 \left([\dim(A)]_q - [k] _q \right) + [k]_q \left( \sum _{A \in \mathcal{F}}   \lambda_A \right)^2. \label{eqn:2} 
\end{equation}
From \eqref{eqn:1} and \eqref{eqn:2}, we get 
\begin{equation}
\label{eqn:3}
     \sum _{A \in \mathcal{F}} \lambda_A ^2 \left([\dim(A)]_q - [k] _q \right) + [k]_q \left( \sum _{A \in \mathcal{F}}   \lambda_A \right)^2= 0.
\end{equation}
We now note that $\dim(A) \geq k$ for all $A \in \mathcal{F}$, and also $\dim(A)=k$ for at most one $A \in \mathcal{F}$, otherwise the intersection condition would not be satisfied. Now, $\dim(A) > k$ implies that $[\dim(A)]_q > [k]_q$ and hence $\lambda_A \neq 0$ for at most one $A \in \mathcal{F}$. 
But in that case $\displaystyle \left( \sum _{A \in \mathcal{F}}   \lambda_A \right)^2$ must be greater than $0$, a contradiction. 

Therefore, the vectors $\{f_A : A \in \mathcal{F} \}$ are linearly independent and hence $|\mathcal{F}| \leq [n]_q$, completing the proof. 
\end{proof}

\section{Proofs of Theorems \ref{thm:oddtown-qanalogue} and \ref{thm:oddtown-qanalogue-rev} }

\begin{proof}[Proof of Theorem \ref{thm:oddtown-qanalogue}] 
Let $q$ be an odd prime power and $\mathcal{F} \subseteq \sub(F_q^n)$ be such that $\dim(A)$ is odd for all $A \in \mathcal{F}$ and $\dim(A \cap B)$ is even for all $A, B \in \mathcal{F}$. Again, as in the proof of Theorem \ref{thm:Fischer-q}, 
for any subspace $A \in \mathcal{F}$, we consider the incidence vector $f_A$, as defined as in \eqref{eqn:defn-inc-vec} and we will show that the vectors $\{f_A : A \in \mathcal{F} \}$ are linearly independent. We assume that there exists a linear relation 
\[ \sum_{A \in \mathcal{F}} \lambda_A f_A= 0_v.\] 
Therefore, for any $B \in \mathcal{F}$, we have 
\begin{equation}
\label{eqn:imp-4}
    \langle \sum_{A \in \mathcal{F}} \lambda_A f_A , f_B \rangle = 0. 
    \end{equation} 
We now observe the following:
\begin{enumerate}
    \item If $A \neq B$ and $A, B \in \mathcal{F}$, by Lemma \ref{lem:imp-1}, we have 
    \begin{equation}
    \label{eqn:imp-5}
        \langle f_A, f_B \rangle = [\dim(A \cap B)]_q= 1+q+\dots+q^{t-1}.
        \end{equation}
        As $q$ is odd and $t$ is even,  $\langle f_A, f_B \rangle$ is even. 
    \item For $A \in \mathcal{F}$, $\langle f_A, f_A \rangle $ is odd, since 
    \begin{equation}
    \label{eqn:imp-6}
        \langle f_A, f_A \rangle = [\dim(A))]_q= 1+q+\dots+q^{\dim(A)-1}   
        \end{equation} 
        and $\dim(A)$ is odd. 
\end{enumerate}
 Using \eqref{eqn:imp-4}, 
 we have
\begin{equation}
   0 =  \langle \sum_{A \in \mathcal{F}} \lambda_A f_A , f_B \rangle  =    \sum_{A \in \mathcal{F}, A \neq B} \lambda _A \langle  f_A , f_B \rangle + \lambda_B \langle f_B, f_B \rangle. 
\end{equation}
By \eqref{eqn:imp-5} and \eqref{eqn:imp-6}, $\lambda_B$ must be zero and thus the vectors $\{f_A : A \in \mathcal{F} \}$
are linearly independent, 
completing the proof. 
\end{proof}

Before proceeding towards 
the proof of Theorem \ref{thm:oddtown-qanalogue-rev}, we need the following lemma \cite[Exercise 1.1.3]{babai-frankl}. 

\begin{lemma}
\label{lem:imp-IJ}
Let $J_m$ be the $m \times m$ matrix with 
all entries $1$ and 
let $I_m$ be
the $m \times m$ identity matrix. 
The rank of $J_m-I_m$ 
over the field $F_2$ is $m$ if $m$ is even and $m-1$ is $m$ is odd.
\end{lemma}




\begin{proof}
[Proof of Theorem \ref{thm:oddtown-qanalogue-rev}]
Let $q$ be odd 
and $\mathcal{F} \subseteq \sub(F_q^n)$ be a family such that $\dim(A)$ is even for all $A \in \mathcal{F}$ and 
$\dim(A \cap B)$ is odd for all $A, B \in \mathcal{F}$. Let $\mathcal{F}=\{A_1, \dots, A_m\}$. 

We consider the $m \times [n]_q$ matrix $M$ whose rows correspond to the members of $\mathcal{F}$, columns correspond 
to the vectors $v_j$ and
the entry $M_{ij}$ is defined as: 
 \begin{equation}
 \label{eqn:defn-inc-mat-M}
 M_{i,j}: = \begin{cases}
1 \hspace{3 mm} \mbox{ if } v_j \in A_i, \\
0 \hspace{3 mm} \mbox{ if } v_j \notin A_i.
\end{cases}
 \end{equation}
 We claim that the rows of $M$ are linearly independent over $F_2$. 

 We consider the matrix $C= M M^T$ where $M^T$ denotes the transpose of $M$. Then,
\begin{equation}
\label{eqn:rank-prod-less-indiv}
 rank(M) \geq rank(C).
 \end{equation}

 From Lemma \ref{lem:imp-1} and Lemma \ref{lem:imp-2}, we see that the matrix $C$ is a $m \times m $ matrix with 
 \begin{equation}
 \label{eqn:defn-inc-mat-A}
 C_{i,j}: = \begin{cases}
 0 \hspace{3 mm} \mbox{ if } i = j , \\
 1  \hspace{3 mm} \mbox{ if } i \neq j .
 \end{cases}
 \end{equation}

Thus, $C= J_m-I_m$ and so we have
$rank(M) \geq rank(MM^T)= rank(J_m-I_m). $
The matrix $M$ clearly has a nontrivial kernel. 
Let $v$ denote the $[n]_q \times 1$ column vector with all ones. Since, the dimension of each subspace of $\mathcal{F}$ is even, by Lemma \ref{lem:imp-1}, each row has an even number of ones. Thus, $Mv \equiv 0_v \mod 2$. So,
\[ rank(J_m - I_m) \leq rank(M) \leq [n]_q-1.  \]
By Lemma \ref{lem:imp-IJ},
$m \leq [n]_q-1$ when $n$ is even and $m \leq [n]_q$ when $n$ is odd.
This completes the proof of the theorem. 
\end{proof}

We now mention the following skew-oddtown theorem \cite{babai-frankl} and a $q$-analogue. 

\begin{theorem}
\label{thm:skew-oddtown}
Let $A_1, \dots, A_m$ and $B_1, \dots, B_m$ be distinct subsets of $[n]$ such that $|A_i \cap B_i|$ is odd
for every $1 \leq i \leq m$ and $|A_i \cap B_j|$ is even for every $i \neq j$. Then, $m \leq n$.  
\end{theorem}

\begin{theorem}
\label{thm:skew-oddtown-q-analogue}
Let $A_1, \dots, A_m$ and $B_1, \dots, B_m$ be distinct subspaces of $F_q^n$ such that $\dim(A_i \cap B_i)$ is odd for every $1 \leq i \leq m$ and 
$\dim(A_i \cap B_j)$ is even for every $i \neq j$. Then, if $q$ is an odd prime power, $m \leq [n]_q$.  
\end{theorem}

As the proof of Theorem \ref{thm:skew-oddtown-q-analogue} follows in an identical way to the proof of Theorem
\ref{thm:oddtown-qanalogue}, we omit the proof. 
We conclude the paper with the following question. 

\begin{question}
What can we say regarding $q$-analogues of oddtown and reverse oddtown theorem when $q$ is a power of $2$?  
\end{question}

\subsection*{Acknowledgements}
The author thanks Angsuman Das for insightful discussions and Narayanan Narayanan for introducing him to the beautiful 
Oddtown Theorem during a workshop a few years ago. The author gratefully acknowledges the support of the NBHM Post Doctoral Fellowship (File No. \linebreak 0204/10(10)/2023/R\&D-II/2781), Government of India, and expresses 
his sincere thanks to the National Board of Higher Mathematics for this funding. The author also appreciates the excellent working environment provided by the Department of Mathematics, Indian Institute of Science.

\end{document}